\documentclass[12pt]{amsart}

\usepackage{amssymb}
\usepackage{soul}
\usepackage{bm}
\usepackage{graphicx}
\usepackage{psfrag}
\usepackage[usenames,dvipsnames]{xcolor}
\usepackage{enumerate}
\usepackage{multirow}
\usepackage{url}

\usepackage{subcaption}
\usepackage{comment}

\usepackage{algpseudocode}

\usepackage{mathtools}

\usepackage{xy}
\input xy
\xyoption{all}

\usepackage{tikz}
\usepackage{tikzscale}
\usetikzlibrary{arrows}

\newcommand{\excise}[1]{}

\newtheorem{thm}{Theorem}[section]

\newtheorem{cor}[thm]{Corollary}
\newtheorem{prop}[thm]{Proposition}

\theoremstyle{definition}

\newtheorem{example}[thm]{Example}

\newtheorem{defn}[thm]{Definition}

\numberwithin{equation}{section}

\renewcommand\>{\rangle}
\newcommand\<{\langle}

\newcommand\QQ{\mathbb{Q}}
\newcommand\RR{\mathbb{R}}

\newcommand\ZZ{\mathbb{Z}}

\newcommand\ww{{\mathbf w}}

\newcommand\set[1]{\{#1\}}

\newcommand{\es}{\emptyset}

\DeclareMathOperator\Betti{Betti} 
\DeclareMathOperator\lcm{lcm} 
\DeclareMathOperator\supp{supp} 
\DeclareMathOperator\Int{Int} 
\DeclareMathOperator\Conv{Conv} 

\DeclareMathOperator\sign{sign} 
\DeclareMathOperator\SIGN{SIGN} 
\DeclareMathOperator\Val{Val} 

\renewcommand{\tt}{\mathbf{t}}
\newcommand{\zz}{\mathbf{z}}

\newcommand{\gen}[1]{\langle{#1}\rangle}

\newcommand{\sm}{\setminus}

\renewcommand{\epsilon}{\varepsilon}

\begin{document}

\mbox{}
\title[Counting edges in factorization graphs]{Counting edges in factorization graphs of numerical semigroup elements}

\author[Moschetti]{Mariah Moschetti}
\address{Mathematics Department\\San Diego State University\\San Diego, CA 92182}
\email{mmoschetti@sdsu.edu}

\author[O'Neill]{Christopher O'Neill}
\address{Mathematics Department\\San Diego State University\\San Diego, CA 92182}
\email{cdoneill@sdsu.edu}

\date{\today}

\begin{abstract}
A numerical semigroup $S$ is an additively-closed set of non-negative integers, and a factorization of an element $n$ of $S$ is an expression of $n$ as a sum of generators of $S$.  It is known that for a given numerical semigroup $S$, the number of factorizations of $n$ coincides with a quasipolynomial (that is, a polynomial whose coefficients are periodic functions of $n$).  One of the standard methods for computing certain semigroup-theoretic invariants involves assembling a graph or simplicial complex derived from the factorizations of $n$.  In this paper, we prove that for two such graphs (which we call the factorization support graph and the trade graph), the number of edges coincides with a quasipolynomial function of $n$, and identify the degree, period, and leading coefficient of each.  In the process, we uncover a surprising geometric connection:\ a combinatorially-assembled cubical complex that is homeomorphic to real projective space.  
\end{abstract}

\maketitle


\section{Introduction}
\label{sec:intro}

Numerical semigroups (additive subsemigroups of the non-negative integers) are algebraic objects that arise naturally in additive combinatorics.  
A classic example of a numerical semigroup is the McNugget semigroup 
$$S = \{6a + 9b + 20c : a, b, c \in \ZZ_{\ge 0}\},$$
named so because McDonald's originally sold chicken McNuggets in packs of size 6, 9, and~20 in the United States~\cite{mcnuggetmag}, so an integer $n$ lies in $S$ when one can buy exactly $n$ McNuggets using these pack sizes.  More generally, we write
$$S = \gen{n_1, \ldots, n_k} = \{z_1n_1 + \cdots + z_kn_k : z_i \in \ZZ_{\geq 0}\}$$
for the smallest numerical semigroup containing  $n_1, \ldots, n_k$, which are referred to as the \emph{generators} of $S$.  

Numerical semigroups are of great interest in factorization theory due in part to the highly non-unique nature of their elements' factorizations~\cite{numericalsurvey}.  Given a numerical semigroup 
$S = \gen{n_1, \ldots, n_k}$
and an element $n \in S$, a \emph{factorization} of $n$ is an expression
$$n = z_1 n_1 + \cdots + z_k n_k$$
of $n$ as a sum of the generators $n_1, \ldots, n_k$ of $S$, and the \emph{set of factorizations} of $n$ is
$$\mathsf{Z}_S(n) = \set{(z_1, \ldots, z_k) \in \ZZ_{\geq 0} : n = z_1n_1 + \cdots + z_kn_k},$$
wherein each factorization is represented by a $k$-tuple.  
In the McNugget semigroup, factorizations of $n$ correspond to the different ways one can purchase $n$ McNuggets, e.g., one can purchase 18 McNuggets with either two packs of $9$ or three packs of $6$, so 
$$\mathsf{Z}_S(18)=\set{(3,0,0), (0,2,0)}.$$

It is natural to wonder, for a given numerical semigroup $S$, what can be said about $|\mathsf Z_S(n)|$, viewed as a function of $n$?  It turns out, $|\mathsf Z_S(n)|$ is a degree $k-1$ quasipolynomial function of $n$; that is,  
assuming $\gcd(n_1, \ldots, n_k)=1$, there exists periodic functions $a_0, \ldots, a_{k-2}: \ZZ_{\geq} \rightarrow \QQ$, each with period dividing $\lcm(n_1, \ldots, n_k)$, such that 
$$
|\mathsf{Z}_{S}(n)|=\frac{1}{(k-1)!n_1\cdots n_k}n^{k-1}+a_{k-2}(n)n^{k-2}+\cdots+a_1(n)n+a_0(n)
$$
for all $n \geq 0$.  This result can be viewed as a consequence of Ehrhart's theorem~\cite{continuousdiscretely,ehrhart}, and is proved in \cite{factorhilbert} using machinery from combinatorial commutative algebra.

One of the standard methods for computing certain semigroup-theoretic invariants (e.g., Betti numbers~\cite{numericalappl} or the catenary degree~\cite{catenarytamefingen}) involves assembling a graph or simplicial complex derived from the factorizations of $n$~\cite{compoverview,affineinvariantcomp}.  Two such graphs that are often used, which we refer to here as the factorization support graph~\cite{numericalappl} and the trade graph~\cite{semitheor} (Definitions~\ref{d:min_trade_graph} and~\ref{d:factgraph}, respectively), both have vertex set $\mathsf{Z}_S(n)$, but with different rules for when two factorizations of $n$ are connected by an edge.  

In this work, we prove that for each graph, the number of edges, viewed as a function of~$n$, coincides with a quasipolynomial for sufficiently large $n$, and identify the degree, period, and leading coefficient (Theorems~\ref{thm:nonedgesF} and~\ref{thm:min_prez_edges}, respectively).  For the factorization support graph, this quasipolynomial has substantially higher degree, meaning for large $n$ it is less efficient for invariant computation than the trade graph.  Additionally, the trade graph is defined in terms of a chosen toric basis~\cite{alggeodiscopt,grobpoly}, for which there are many standard choices that vary in size~\cite{toricbasesrelativesize}.  For instance, the Markov basis~\cite{markovbook,algmarkov} has minimal cardinality, while the universal Gr\"obner basis and the Graver basis~\cite{gravercomplexitymatrices,groebnercomplexitymatrices,groebnerminimaldegree} are notoriously larger.  By Theorem~\ref{thm:min_prez_edges}, regardless of which toric basis is chosen, the number of edges in the trade graph is a quasipolynomial of the same degree, and only the leading coefficient depends on the number of trades.  

We note briefly that there is a third graph appearing frequently in the literature, known as the squarefree divisor graph~\cite{squarefreedivisorcomplex,rosalesrelationalg,squarefreedivisororigin}, which has one vertex for each generator of~$S$ and thus is generally much smaller the trade graph or factorization support graph.  This graph can be used to compute minimal presentations and Betti numbers of $S$, but is insufficient for other invariants like the catenary degree.  Regardless, we do not examine the squarefree divisor graph in this paper as it is known to coincide with the complete graph for all but finitely many $n \in S$.   

In the process of proving Theorem~\ref{thm:nonedgesF}, we uncover a cubical complex that is homeomorphic to real projective space; we explore this connection in 
Section~\ref{sec:cubecomplexes}.  

\section{Background and notation}
\label{sec:background}

Given $k \in \ZZ_{>0}$, we write $[k] = \set{1, \ldots, k}$. 
For ease of notation, and when no confusion can arise, we sometimes refer to a subset of $[k]$ by its elements, e.g., writing $\{13, 245\}$ for $\{\{1,3\}, \{2,4,5\}\}$.  

Our notation and terminology for graphs and posets follows that of~\cite{walkthroughcomb}; we refer the reader to this monograph for any missing definitions.  Given a finite poset $(P,\preceq)$, we call sets of the form
$$[x,y) = \{z \in P : x \preceq z \preceq y, \, z \ne y\},$$
\emph{intervals} of $P$, and $\Int(P)$ denotes the set of intervals of $P$.  
We briefly recall here the \emph{M\"{o}bius function} $\mu:\Int(P) \rightarrow \RR$, which has $\mu(x,x)=1$ and 
$$\mu(x,y)=-\!\!\sum_{z\in[x,y)}\mu(x,z)$$
for all $x \prec y$, as well as the \emph{Dual M\"{o}bius Inversion Formula}, stated below.  

\begin{thm}[{\cite[Corollary~16.22]{walkthroughcomb}}]\label{thm:MIF}
Let $P$ be a finite poset, and let $f: P \rightarrow \RR$ be a function. Let the function $g:P\rightarrow \RR$ be defined by 
$$g(y)=\sum_{x \succeq y}f(x).$$
Then
$$f(y)=\sum_{x \succeq y}g(x)\mu(y,x).$$
\end{thm}

In the remainder of this section, we establish the necessary background and notation on numerical semigroups. For a more thorough introduction, see \cite{numericalappl}. 

The \emph{numerical semigroup} generated by $n_1, \ldots, n_k \in \ZZ_{\ge 1}$ is the additive semigroup
$$S = \gen{n_1, \ldots, n_k} = \set{a_1n_1 + \cdots + a_kn_k : a_1, \ldots, a_k \in \ZZ_{\ge 0}}.$$
Note that we do not require $n_1, \ldots, n_k$ to be relatively prime, nor, equivalently, that $\ZZ_{\ge 0} \setminus S$ be finite.  If no proper subset of $n_1, \ldots, n_k$ generates the same subset of $\ZZ_{\ge 0}$, then we say $n_1, \ldots, n_k$ is the \emph{minimal generating set} of $S$.
Every numerical semigroup has a unique minimal generating set, and unless otherwise specified, whenever we write $S=\gen{n_1, \ldots ,n_k}$, we assume $n_1, \ldots, n_k$ is the minimal generating set of $S$. The elements $n_1, \ldots, n_k$ are then referred to as the \emph{atoms} of $S$, and $\mathsf e(S) = k$ as the \emph{embedding dimension} of $S$.  
A \emph{factorization} of $n \in S$ is an expression of the form 
$$n = z_1n_1 + \cdots + z_kn_k$$
for some $\mathbf z = (z_1, \ldots, z_k) \in \ZZ_{\ge 0}^k$, and the \emph{set of factorizations} of $n$ is 
$$\mathsf{Z}_{S}(n)=\set{(z_1,\ldots,z_k) : n = z_1n_1 + \cdots +z_kn_k}.$$
We also often refer to elements $\zz \in \mathsf Z_S(n)$ as \emph{factorizations} of $n$.  When there is no confusion which numerical semigroup is being considered, we simply write $\mathsf{Z}(n)$ for $\mathsf{Z}_{S}(n)$.  Given $\zz \in \mathsf Z(n)$, we write $\supp(\zz) = \{i : z_i > 0\}$, and given $I \subseteq [k]$, we write $\mathsf{Z}_I(n)$ for $\mathsf{Z}_T(n)$, where $T = \gen{n_i : i \in I}$. 

We conclude this section with a description of the function $n \mapsto \mathsf |Z_S(n)|$ for fixed~$S$.  
This result was stated as \cite[Theorem~3.9]{factorhilbert} in the special case where the generators of $S$ are relatively prime, though the result itself is folklore.  

A function $f: \ZZ_{\ge 0} \rightarrow \RR$ is a \emph{quasipolynomial} of degree $d$ if there exists periodic functions $a_0,\ldots,a_d: \ZZ_{\geq 0} \rightarrow \RR$, with $a_d$ not identically 0, such that $$f(n) =a_d(n)n^d+ \cdots +a_1(n)n+a_0(n).$$
The \emph{period} of $f$ is the minimal integer $p$ such that $a_i(n+p) =a_i(n)$ for all $ 0 \leq i \leq d$.

\begin{thm}\label{thm:Hilbert_FactIsQuasi}
Fix a numerical semigroup $S = \gen{n_1,\ldots,n_k}$ with $\gcd(n_1, \ldots, n_k)=c$. Then there exist periodic functions $a_0,\ldots,a_{k-2}: \ZZ_{\geq0} \rightarrow \QQ$, each with period dividing $\lcm{(n_1, \ldots, n_k)}$, such that if $n \ge 0$ and $c \mid n$, then 
$$|\mathsf{Z}_{S}(n)|=\frac{c}{(k-1)! \, n_1\cdots n_k}n^{k-1}+a_{k-2}(n)n^{k-2}+\cdots+a_1(n)n+a_0(n),$$
and $|\mathsf{Z}_{S}(n)| = 0$ otherwise.
\end{thm}

\begin{proof}
Let $S' = \gen{\tfrac{n_1}{c}, \ldots, \tfrac{n_k}{c}}$. Given any $n \in S$, we have $n/c \in S'$. Moreover, given any $n' \in S'$ such that $n'=n/c$, we have $n \in S$. Thus, it is easy to see $\mathsf{Z}_{S}(n) = \mathsf{Z}_{S'}(n/c)$. 

The generators of $S'$ are coprime by definition. By \cite[Theorem~3.9]{factorhilbert}, $|\mathsf Z_S(n)|$ coincides with a degree $k-1$ quasipolynomial with period dividing $\lcm{(\tfrac{n_1}{c}, \ldots, \tfrac{n_k}{c})}$ for all $n \geq 0$, so 
$$
|\mathsf{Z}_{S}(n)| = 
\begin{cases}
|\mathsf{Z}_{S'}(\tfrac{n}{c})| & c \mid n; \\
0 & c \nmid n,
\end{cases}
$$
coincides with a degree $k-1$ quasipolynomial with period dividing $\lcm{(n_1, \ldots, n_k)}$ for all $n \geq 0$.  Moreover, restricting the domain of $|\mathsf{Z}_{S}(n)|$ to $n \in c\ZZ_{\ge 0}$, by \cite[Theorem~3.9]{factorhilbert} the leading coefficient is
\begin{align*}
\displaystyle\frac{1}{(k-1)! \, \tfrac{n_1}{c}\cdots \tfrac{n_k}{c}}(n/c)^{k-1} &= 
\displaystyle\frac{c}{(k-1)! \, n_1\cdots n_k}n^{k-1}.
\end{align*}
This gives the claimed leading coefficient whenever $c \mid n$, as $|\mathsf{Z}_{S}(n)| = 0$ otherwise.
\end{proof}

\section{Counting edges in factorization graphs}
\label{sec:countingedges}

In this section, we examine two graphs associated to each element $n \in S$, both with vertex set $\mathsf Z_S(n)$.  We prove the number of edges in each is a quasipolynomial function of $n$, and identify its degree, period, and leading coefficient.

\subsection{Factorization support graphs}

The factorization support graph of an element $n \in S$ is a graph that relates the factorizations of $n$ by which generators of $S$ they share~\cite{numericalappl}.  
For an example, see Figure \ref{fig:FSG_example1}. 

\begin{defn}\label{d:factgraph}
The \emph{factorization support graph of $n$} is the simple graph $\nabla_S(n)$ with vertex set $\mathsf{Z}_{S}(n)$ and an edge between $\zz$ and $\zz'$ whenever $\supp(\zz) \cap \supp(\zz') \neq \es$. 
\end{defn}

Let us briefly examine the case $k = 3$.  

\begin{figure}[t]
\centering
\begin{tikzpicture}
\node[shape=circle,draw=black,label={above:$(4,0,1)$}] (A) at (0,0) {};
\node[shape=circle,draw=black, label={above:$(1,0,3)$}] (B) at (2,0){};
\node[shape=circle,draw=black, label={above:$(0,4,0)$}] (C) at (4,0){};
\path [-] (A) edge node[right] {} (B);
\end{tikzpicture}
\caption{The factorization support graph $\nabla_S(44)$ for $S=\gen{8,11,12}$.}
\label{fig:FSG_example1}
\end{figure}
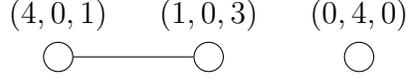

\begin{example}\label{ex:k=3nonedgecount-FSG}
For $S= \gen{n_1, n_2, n_3}$ and positive $n \in S$, two factorizations $\zz, \zz' \in \mathsf Z_S(n)$ lack an edge in $\nabla_S(n)$ whenever $\supp(\zz) \cap \supp(\zz') = \emptyset$.  As such, neither $\zz$ nor $\zz'$ can have full support, so $\{\supp(\zz), \supp(\zz')\}$ must equal 
$$
\set{12,3},
\quad
\set{13, 2},
\quad
\set{23, 1},
\quad
\set{1, 2},
\quad
\set{1,3},
\quad
\text{or}
\quad \set{2,3}.
$$
Note that if $\{q,r,s\} = [3]$, then $\zz \in \mathsf{Z}_{q}(n)$ implies $\zz \in \mathsf{Z}_{qr}(n)$ and $\zz \in \mathsf{Z}_{qs}(n)$.  A~quick inclusion-exclusion argument gives us exactly
\begin{align*}
\binom{|\mathsf{Z}_{S}(n)|}{2}
&- |\mathsf{Z}_{12}(n)||\mathsf{Z}_{3}(n)| - |\mathsf{Z}_{13}(n)||\mathsf{Z}_{2}(n)| - |\mathsf{Z}_{23}(n)||\mathsf{Z}_{1}(n)|
\\
&+ |\mathsf{Z}_{1}(n)||\mathsf{Z}_{2}(n)| + |\mathsf{Z}_{1}(n)||\mathsf{Z}_{3}(n)| + |\mathsf{Z}_{2}(n)||\mathsf{Z}_{3}(n)|
\end{align*}
total edges in $\nabla_S(n)$.
\end{example}

The remainder of this subsection is devoted to obtaining a formula (Theorem~\ref{thm:nonedgesF}) for the number of edges of $\nabla_S(n)$ using an inclusion-exclusion argument in the spirit of Example~\ref{ex:k=3nonedgecount-FSG}.  
To do this precisely, we construct the following poset and apply the Dual M\"{o}bius Inversion Formula (Theorem~\ref{thm:MIF}).

\begin{defn}\label{d:posetds}
The \emph{poset of disjoint supports on $k$ elements} is the partially ordered set $(\mathcal{P}_k, \preceq)$ whose ground set is the unordered pairs $\set{I,J}$ such that $I, J \subseteq [k]$ are nonempty, proper disjoint subsets, together with a unique minimum element $\hat 0$, such that $\set{R,T} \preceq \set{I, J}$ whenever, up to relabeling, $R \supseteq I$ and $T \supseteq J$.
\end{defn}

\begin{example}
Notice that $\mathcal{P}_3$, depicted in Figure~\ref{fig:posetds_k=3}, encapsulates the structure of the inclusion-exclusion argument in Example~\ref{ex:k=3nonedgecount-FSG}.  
In particular, for each $\set{I,J} \in \mathcal{P}_3$, the value $\mu(\hat0,\set{I,J})$ (depicted in red next to $\{I,J\}$ in Figure~\ref{fig:posetds_k=3}) equals the coefficient of the term $|\mathsf{Z}_I(n)||\mathsf{Z}_J(n)|$ in the expression for the number of edges in $\nabla_S(n)$.  
\end{example}

\begin{figure}[t]
    \centering
\begin{tikzpicture}


\draw[thick] (2,0) -- (4.5,2);

\draw[thick] (2,0) -- (2,2);

\draw[thick] (2,0) -- (-0.5,2);

\draw[thick] (-0.5,2) -- (2,4);

\draw[thick] (-0.5,2) -- (-0.5,4);

\draw[thick] (2,2) -- (4.5,4);

\draw[thick] (2,2) -- (-0.5,4);

\draw[thick] (4.5,2) -- (4.5,4);

\draw[thick] (4.5,2) -- (2,4);

\filldraw (2,0) circle (2pt) node[anchor=north] {$\hat0$} node[anchor=west,  yshift=-1mm, color=red] {$1$};

\filldraw (-0.5,2) circle (2pt) node[anchor=north,xshift=-4mm]{$\set{12,3}$} node[anchor=south east, xshift=-1mm, color=red] {$\textrm{-}1$\!};
\filldraw (2,2) circle (2pt) node[anchor=north,xshift=-1.5mm] {$\set{13,2}$} node[anchor=south, yshift=1mm, color=red] {$\textrm{-}1$};
\filldraw (4.5,2) circle (2pt) node[anchor=north,xshift=4mm] {$\set{23,1}$} node[anchor=south west, color=red] {$\textrm{-}1$};

\filldraw (4.5,4) circle (2pt) node[anchor=south] {$\set{1,2}$} node[anchor=north west, color=red] {$1$};
\filldraw (2,4) circle (2pt) node[anchor=south] {$\set{1,3}$} node[anchor=north, color=red] {$1$};
\filldraw (-0.5,4) circle (2pt) node[anchor=south] {$\set{2,3}$} node[anchor=north east, color=red] {$1$};
\end{tikzpicture}
    \caption{The poset $\mathcal{P}_3$ and the values of $\mu(\hat0, -)$ in red.}
    \label{fig:posetds_k=3}
\end{figure}
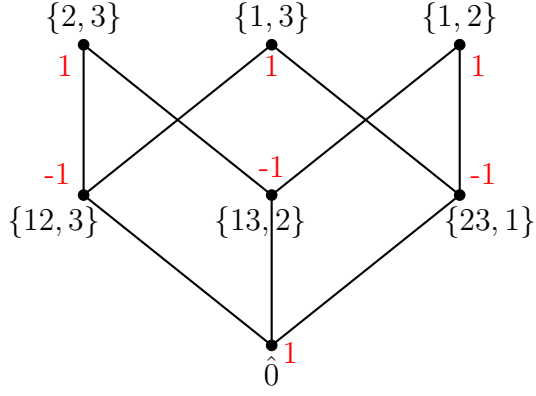

\begin{prop}\label{prop:mobius_PDS}
For any $\set{I,J} \in \mathcal{P}_k$, we have
$$\mu(\hat 0, \set{I,J}) = (-1)^{k-|I|-|J|+1}.$$
\end{prop}

\begin{proof}
Fix $\set{I,J} \in \mathcal{P}_k \setminus \{\hat 0\}$, and let $m = k-|I|-|J|$ (the height of $\set{I,J}$ in $\mathcal P_k$).  If~$m=0$, then $[k] = I \sqcup J$, so $\{I,J\}$ covers $\hat 0$, and thus $\mu(\hat0, \set{I,J}) = -1 = (-1)^{m+1}.$  As such, suppose $m > 0$.  

Suppose $\{R,T\} \in (\hat 0, \{I,J\})$, and let 
$i = |R| + |T| - |I| - |J| \ge 0.$
Since $R \supseteq I$ and $T \supseteq J$, each such element $\set{R,T}$ is uniquely determined by a choice of $i$ elements from $[k] \sm (I \sqcup J)$ and a partition of those $i$ elements into $R \setminus I$ and $T \setminus J$.  There are thus $2^i\binom{m}{i}$ such elements.  Proceeding by induction on $m$, we may assume 
$$\mu(\hat0,\set{R,T})=(-1)^{k-|R|-|T|+1}= (-1)^{k-|I|-|J|-i+1} = (-1)^{m-i+1}$$
whenever $\{R,T\} \prec \{I,J\}$, meaning
\begin{align*}
\mu(\hat0,\set{I,J})
&= -\mu(\hat0,\hat0)-\!\!\!\!\!\!\!\!\!\sum_{\set{R,T}\in(\hat0, \set{I,J})}\!\!\!\!\!\!\!\!\!\mu(\hat0,\set{R,T})
= -1 - \sum_{i=1}^{m}(-1)^{m-i+1}{2^i\binom{m}{i}} \\
&= -1+\sum_{i=1}^{m}(-1)^{m-i}{2^i\binom{m}{i}}
= -1 +(2-1)^{m} +(-1)^{m+1}
=(-1)^{m+1},
\end{align*}
applying the binomial theorem in the last line.  
\end{proof}

\begin{thm}\label{thm:nonedgesF}
The number of edges in $\nabla_S(n)$ is given by
$$
|E(\nabla_S(n))| = 
\binom{|\mathsf{Z}_{S}(n)|}{2} + \sum_{\set{I,J} \in \mathcal P_k} (-1)^{k-|I|-|J|+1} |\mathsf{Z}_I(n)||\mathsf{Z}_J(n)|.
$$
In particular, $|E(\nabla_S(n))|$ is a quasipolynomial of degree $2k - 2$ with leading coefficient 
$$\frac{c^2}{2\left((k-1)!\right)^2n_1^2 \cdots n_k^2}$$
and period dividing $\lcm(n_1, \ldots, n_k)$, where $c = \gcd(n_1, \ldots, n_k)$.  
\end{thm}

\begin{proof}
Define $f: \mathcal{P}_k \rightarrow \ZZ_{\ge 0}$ with $f(\hat0) = |E(\nabla_S(n))|$ and $f(\set{I,J})$ equaling the number of nonedges $\set{\zz,\zz'}$ in $\nabla_S(n)$ such that $\supp(\zz) = I$ and $\supp(\zz') = J$.  Define $g:\mathcal{P}_k \rightarrow \ZZ_{\ge 0}$ such that
$$
g(\hat0)
= f(\hat0)+\sum_{\set{I,J} \succ \hat0} f(\set{I,J})
= \binom{|\mathsf{Z}_{S}(n)|}{2}
$$
equals the total number of edges in the complete graph on $|\mathsf{Z}_{S}(n)|$ vertices, and 
$$
g(\{I,J\})
= \sum_{X \succeq \{I,J\}} f(X)
= |\mathsf{Z}_I(n)||\mathsf{Z}_J(n)|$$
equals the number of nonedges $\set{\zz, \zz'}$ in $\nabla_S(n)$ such that $\supp(\zz) \subseteq I$ and $\supp(\zz') \subseteq J$.  
By Theorem \ref{thm:MIF} and Proposition \ref{prop:mobius_PDS}, we have
$$
|E(\nabla_S(n))|
= f(\hat0)
=\sum_{X \succeq \hat0}g(X)\mu(\hat0,X)
= \binom{|\mathsf{Z}_{S}(n)|}{2} + \sum_{\set{I,J} \succ \hat0} (-1)^{k-|I|-|J|+1} |\mathsf{Z}_I(n)||\mathsf{Z}_J(n)|.
$$
The remaining claims then follow from the fact that 
$$
\binom{|Z_S(n)|}{2}
\qquad \text{and} \qquad
|\mathsf{Z}_I(n)||\mathsf{Z}_J(n)|
$$
are quasipolynomials by Theorem~\ref{thm:Hilbert_FactIsQuasi}, each with period dividing $\lcm(\tfrac{n_1}{c}, \ldots, \tfrac{n_k}{c})$, and the former attaining maximum degree in the expression for $|E(\nabla_S(n))|$.  
\end{proof}

\subsection{Trade graphs}

The second family of graphs requires a few additional definitions.  
The \emph{factorization homomorphism of $S$} to be the map $\varphi:~{\ZZ}_{\geq 0}^k \rightarrow S$ given by
$$\varphi\left(z_1,\ldots,z_k\right) = z_1n_1 + \cdots + z_kn_k,$$
which sends each $k$-tuple $\zz \in \mathsf Z_S(n)$ to the element $n \in S$ it is a factorization of.  In~particular, $\mathsf Z_S(n) = \varphi^{-1}(n)$ for each $n \in S$.  We have
$$\ker\varphi = \set{(\zz,\zz') \in {\ZZ}_{\geq 0}^k \times {\ZZ}_{\geq 0}^k : \varphi(\zz)=\varphi(\zz')}.$$ 
The kernel of $\varphi$ is a \emph{congruence}; that is, an equivalence relation respecting \emph{translation}, that is, if $(\zz, \zz') \in \ker\varphi$, then $(\zz + \mathbf{v}, \zz' + \mathbf{v}) \in \ker\varphi$ for all $\mathbf{v} \in {\ZZ}_{\geq 0}^k$.  The elements of $\ker\varphi$ are called \emph{trades} of $S$.  
Given $\mathbf{z}, \mathbf{z}' \in \mathsf{Z}_S(n)$, we define the \emph{greatest common divisor} of $\mathbf{z}$ and $\mathbf{z}'$ to be $$\gcd(\zz,\zz')= \left(\min(z_1,z'_1), \ldots, \min(z_k,z'_k)\right).$$
The trade $(\zz-\gcd(\zz,\zz'),\zz'-\gcd(\zz,\zz'))$ is called the \emph{direct trade between $\zz$ and $\zz'$}.  

Given a numerical semigroup $S$, a \emph{presentation of $S$} is a set $\rho \subseteq \ker\varphi$ such that for each element $n\in S$, any two factorizations $\zz,\zz' \in \mathsf{Z}(n)$ are \emph{connected by a sequence of trades} in $\rho$, that is, 
there exists a sequence $\ww_0, \ww_1, \ldots, \ww_r \in \mathsf Z_S(n)$ such that $\ww_0 = \zz$, $\ww_r = \zz'$, and for each $i$, at least one of the the direct trades $(\ww_i,\ww_{i+1})$ and $(\ww_{i+1}, \ww_i)$ lies in $\rho$.  
A presentation $\rho$ of $S$ is called \emph{minimal} if it is minimal with respect to containment among all presentations of $S$.  It is known that all minimal presentations of $S$ are finite and have equal cardinality.  If $\rho$ is minimal, the elements of the set
$$\Betti(S) = \{\varphi(\tt_i) : (\tt_i, \tt_i') \in \rho\}$$
are called \emph{Betti elements}, and do not depend on the choice of minimal presentation \cite{numericalappl}.

\begin{defn}\label{d:min_trade_graph}
Fix a presentation $\rho$ of $S$. The \emph{trade graph of $n$ with respect to $\rho$} is the simple graph $\mathrm{T}_{S, \, \rho}(n)$ with vertex set $\mathsf{Z}_{S}(n)$ and an edge connecting two vertices $\zz$ and $\zz'$ whenever the direct trade between $\zz$ and $\zz'$ lies in $\rho$.  
\end{defn}

\begin{example}\label{ex:trades_example}
Let $S = \<6, 9, 20\>$.  
It turns out that any two factorizations of $n \in S$ are connected by a sequence of trades in the set 
$$\rho = \set{((3,0,0),(0,2,0)),((4,4,0),(0,0,3))},$$
and omitting either trade causes this to fail for $n = 18$ and $n = 60$, respectively, making $\rho$ a minimal presentation of $S$.  As such, $\Betti(S) = \set{18,60}$.  In general, both $\rho$ and the set $\Betti(S)$ can be computed as detailed in~\cite{mcnuggetdistances}.  

Let us briefly examine $\mathsf Z_S(78)$, whose elements are depicted as vertices in Figure~\ref{fig:tradeGraph}.  Two factorizations therein are connected by a thick green edge if the direct trade between them lies in $\rho$ (horizontal for $((3,0,0),(0,2,0))$, and vertical otherwise).  Since $\rho$ is a presentation of $S$, the resulting graph is assuredly connected.   If we instead used the (non-minimal) presentation
$$\rho' = \set{((3,0,0),(0,2,0)),((4,4,0),(0,0,3)), ((7,2,0),(0,0,3))}$$
for $S$, the remaining edges in Figure~\ref{fig:tradeGraph} (both diagonal) would also be included.  
\end{example}

\begin{figure}[t]
\centering
\begin{tikzpicture}
\node[shape=circle,draw=black, label={below:$(13,0,0)$}] (A) at (-4,0){};
\node[shape=circle,draw=black, label={below:$(10,2,0)$}] (B) at (-2,0){};
\node[shape=circle,draw=black, label={below:$(7,4,0)$}] (C) at (0,0){};
\node[shape=circle,draw=black, label={below:$(4,6,0)$}] (D) at (2,0){};
\node[shape=circle,draw=black, label={below:$(1,8,0)$}] (E) at (4,0){};
\node[shape=circle,draw=black, label={right:$(0,2,3)$}] (F) at (2,2){};
\node[shape=circle,draw=black, label={left:$(3,0,3)$}] (G) at (0,2){};
\path [-] (A) edge[Green, very thick] node[right] {} (B);
\path [-] (B) edge[Green, very thick] node[right] {} (C);
\path [-] (C) edge[Green, very thick] node[right] {} (G);
\path [-] (C) edge[Green, very thick] node[right] {} (D);
\path [-] (G) edge[Green, very thick] node[right] {} (F);
\path [-] (F) edge[Green, very thick] node[right] {} (D);
\path [-] (D) edge[Green, very thick] node[right] {} (E); 

\path [-] (F) edge node[right] {} (C); 
\path [-] (B) edge node[right] {} (G); 
\end{tikzpicture}
\caption{Trade graphs with respect to $\rho'$ (the full graph) and $\rho$ (the green subgraph) in Example \ref{ex:trades_example}.}
\label{fig:tradeGraph}
\end{figure}
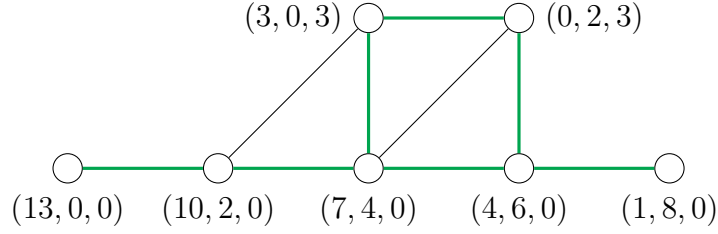

\begin{thm}\label{thm:min_prez_edges}
Fix $S = \gen{n_1, \ldots, n_k}$ and a finite presentation $\rho = \{(\tt_1,\tt'_1),\ldots, (\tt_{r},\tt'_{r})\}$ of $S$.  For each $(\tt_i,\tt'_i) \in \rho$, let $\beta_i = \varphi(\tt_i)$. 
If $n \in S$ with $n \geq \max(\beta_1, \ldots, \beta_r)$, then 
the number of edges in the trade graph $\mathrm{T}_{S, \, \rho}(n)$ is given by 
$$|E(\mathrm{T}_{S, \, \rho}(n))| = \sum_{i=1}^{r} |\mathsf{Z}(n-\beta_i)|.$$
In particular, $|E(\mathrm{T}_{S, \, \rho}(n))|$ coincides for large $n$ with a degree $k-1$ quasipolynomial with leading coefficient 
$$\frac{c|\rho|}{(k-1)!n_1 \cdots n_k}$$
and period dividing $\lcm(n_1, \ldots, n_k)$, where $c = \gcd(n_1, \ldots, n_k)$.  
\end{thm}

\begin{proof}
Suppose $\zz, \zz' \in \mathsf Z_S(n)$ share an edge in $\mathrm T_{S,\rho}(n)$.  Writing $\mathbf u = \gcd(\zz, \zz')$, it is clear that $\gcd(\zz - \mathbf u, \zz' - \mathbf u) = 0$, and in fact, after possibly relabeling $\zz$ and $\zz'$, 
$$(\zz - \mathbf u, \zz' - \mathbf u) = (\tt_i, \tt_i') \in \rho$$
for some $i$.  The key observation is then that each edge in the graph $\mathrm T_{S,\rho}(n)$ is uniquely determined by a choice of $i \in [1, r]$ and a factorization $\mathbf u \in \mathsf Z_S(n - \beta_i)$.  
Therefore, the total number $|E(\mathrm{T}_{S, \, \rho}(n))|$ of edges in the trade graph $\mathrm{T}_{S, \, \rho}(n)$ is 
$$
|E(\mathrm{T}_{S, \, \rho}(n))| = \sum_{i=1}^{r}|\mathsf{Z}_S(n-\beta_i)|.
$$
The remaining claims now follow from Theorem~\ref{thm:Hilbert_FactIsQuasi}.  
\end{proof}

\begin{example}\label{e:repeatedbetti}
In the statement of Theorem~\ref{thm:min_prez_edges}, it is possible that $\beta_i = \beta_j$ for $i \ne j$.  For example,  if $S = \<6,10,15\>$, then $\beta_1 = \beta_2 = 30$ and $\Betti(S) = \{30\}$, even though any minimal presentation $\rho$ of $S$ has $|\rho| = 2$.  
\end{example}

\section{An interesting cube complex arises}
\label{sec:cubecomplexes}

In Definition~\ref{d:posetds}, we introduced the poset $\mathcal{P}_k$ of disjoint supports on $k$ elements, which was designed to encapsulate the structure of counting the nonedges in $\nabla_S(n)$.  In this section, we provide a geometric interpretation of $\mathcal{P}_k$ in Corollary~\ref{c:real_projective_space}.  
This draws on topics from the geometry of polytopes; we give a brief overview of the definitions and notation used, based on~\cite{ziegler} unless otherwise stated.

A \emph{polytope} $P$ is the convex hull of a finite set of points $K = \set{\mathbf{x_1}, \ldots, \mathbf{x_n}} \subseteq \RR^d$, i.e.,
$$
P=\Conv{(K)} = \set{\lambda_1\mathbf{x_1} + \cdots + \lambda_n\mathbf{x_n} : \lambda_i \in \RR_{\ge 0}, \, \lambda_1 + \cdots + \lambda_n = 1}.
$$ 
The \emph{dimension} of $P$ is the $\RR$-vector space dimension of its affine hull.  A \emph{face} of $P$ is a set of the form $F = P \cap H$ where $H$ is the boundary of a half-space containing $P$.  The~0-dimensional faces of $P$ are called its \emph{vertices}, and their convex hull equals $P$.  
The~\emph{face lattice} of $P$ is the poset $L(P)$ of all faces of $P$, partially ordered by inclusion.

Following~\cite{canonicalcutshypercube}, the \emph{$d$-dimensional unit hypercube} (or \emph{$d$-cube}) is the polytope 
$$C_d = \set{(x_1, \ldots, x_d) \in \RR^d : 0 \leq x_j \leq 1}$$
whose vertex set $V_d$ consists of the points $\mathbf{x}$ such that each $x_j \in \set{0,1}$. Each face $F$ of $C_d$ with dimension $r$ is a set of $\mathbf{x} \in C_d$ satisfying exactly $d-r$ equations of the form 
$$
x_{j_1} = \cdots = x_{j_i} = 0 
\qquad \text{or} \qquad
x_{j_{i+1}} = \cdots = x_{j_{d-r}} = 1
$$
where $I = \set{j_1,\ldots, j_i}$ and $J = \set{j_{i+1}, \ldots, j_{d-r}}$ are disjoint.
The set $[d]$ can be partitioned into 
\begin{align*}
N(F)^{+} &= \set{j \in \set{1,\ldots, d} : x_j = 1 \text{ for all } \mathbf{x} \in V_d \cap F},\\
N(F)^{-} &= \set{j \in \set{1,\ldots, d} : x_j = 0 \text{ for all } \mathbf{x} \in V_d \cap F}, \text{ and}\\
N(F)^0 &= \set{1,\ldots, d} \sm \left(N(F)^{+} \cup N(F)^{-}\right).
\end{align*}
based on the coordinates of the vertices of $F$.  By \cite[Corollary~7.17]{ziegler}, any choice of 3 disjoint subsets $I, J, K \subseteq [k]$ whose union is $[k]$ uniquely determines a face $F$ with $N(F)^+ = I$, $N(F)^- = J$, and $N(F)^0 = K$.  

\begin{defn}\label{d:ordered_poset}
The \emph{ordered poset of disjoint supports on $k$ elements} is the poset $(P_k, \preceq)$ 
whose ground set is the ordered pairs $(I,J)$ such that $I, J \subseteq [k]$ are nonempty, proper disjoint subsets, together with a unique minimum element $\hat 0$, such that 
$(R,T) \preceq (I,J)$ when $I \subseteq R$ and $J \subseteq T$.
\end{defn}

\begin{prop}\label{prop:ideal_is_cube}
For any distinct $m, n \in [k]$, the principal ideal 
$$P = \set{ (I,J) :  \set{I, J} \in P_{m,n}, m \in I, n \in J} \cup \set{\hat 0}$$
generated by the maximal element $(m,n) \in P_k$ is isomorphic as a poset to the face lattice $L(C_{k-2})$, as is the principal ideal generated by $\set{m,n}$ in $\mathcal P_k$.  
\end{prop}

\begin{proof}
Up to permutation of $[k]$, we may assume $m = k-1$ and $n = k$, so that each $(I,J) \in P_{m,n}$ has $I, J \subseteq [k-2]$.  
Let $\phi: P_{m,n} \rightarrow L(C_{k-2})$ be the map with $\hat 0 \mapsto \es$ that sends each $(I,J)$ to the face $F \in L(C_{k-2})$ with $N(F)^- = I$ and $N(F)^+ = J$.  This map is a bijection by \cite[Corollary~7.17]{ziegler}, and whenever $(R, T) \preceq (I,J)$, each facet containing $\phi(I,J)$ also contains $\phi(R,T)$, so $\phi$ is order-preserving.  

Lastly, Considering the order-preserving map $P_k \to \mathcal P_k$ sending each $(I,J) \mapsto \{I,J\}$, the final claim then follows from the fact that each $\{I,J\} \preceq \{m,n\}$ is the image of exactly one element of $P_{m,n}$.  
\end{proof}

Proposition \ref{prop:ideal_is_cube} characterizes each principal ideal generated by a maximal element of $\mathcal P_k$ as the face lattice of a $(k-2)$-cube.  Since every face of a cube is itself a cube, every principal ideal in $\mathcal P_k$ is isomorphic to the face lattice of a cube.  Additionally, the intersection of two principal ideals (say with generators $(I,J)$ and $(I',J')$, respectively), is again a principal ideal (namely the one generated by $(I \cap I', J \cap J')$).  This endows $P_k$ with the structure of a \emph{cubical complex}; that is, a collection $\mathcal C$ of cubes satisfying (i) any face of a cube $C \in \mathcal C$ is again a cube in $\mathcal C$, and (ii) the intersection of any two cubes in $\mathcal C$ again lies in~$\mathcal C$.  The~same can also be said for $\mathcal P_k$ by Proposition~\ref{prop:ideal_is_cube}.  

\begin{example}\label{e:cubical_complex}
The principal ideal generated by any maximal element $\set{m,n} \in \mathcal{P}_4$ is isomorphic to the face lattice of a square.  Viewing $\mathcal P_4$ as a cubical complex provides a recipe to ``glue" all such squares together along edges/vertices whenever those principal ideals coincide in $\mathcal P_4$.  This results in the cubical complex depicted on the left in Figure~\ref{fig:cubical_complex}, with opposing ``outside'' edges identified (e.g., $(3,14)$ and $(14,3)$).  
\end{example}

\begin{figure}[t]
\centering
\includegraphics[width=2.5in]{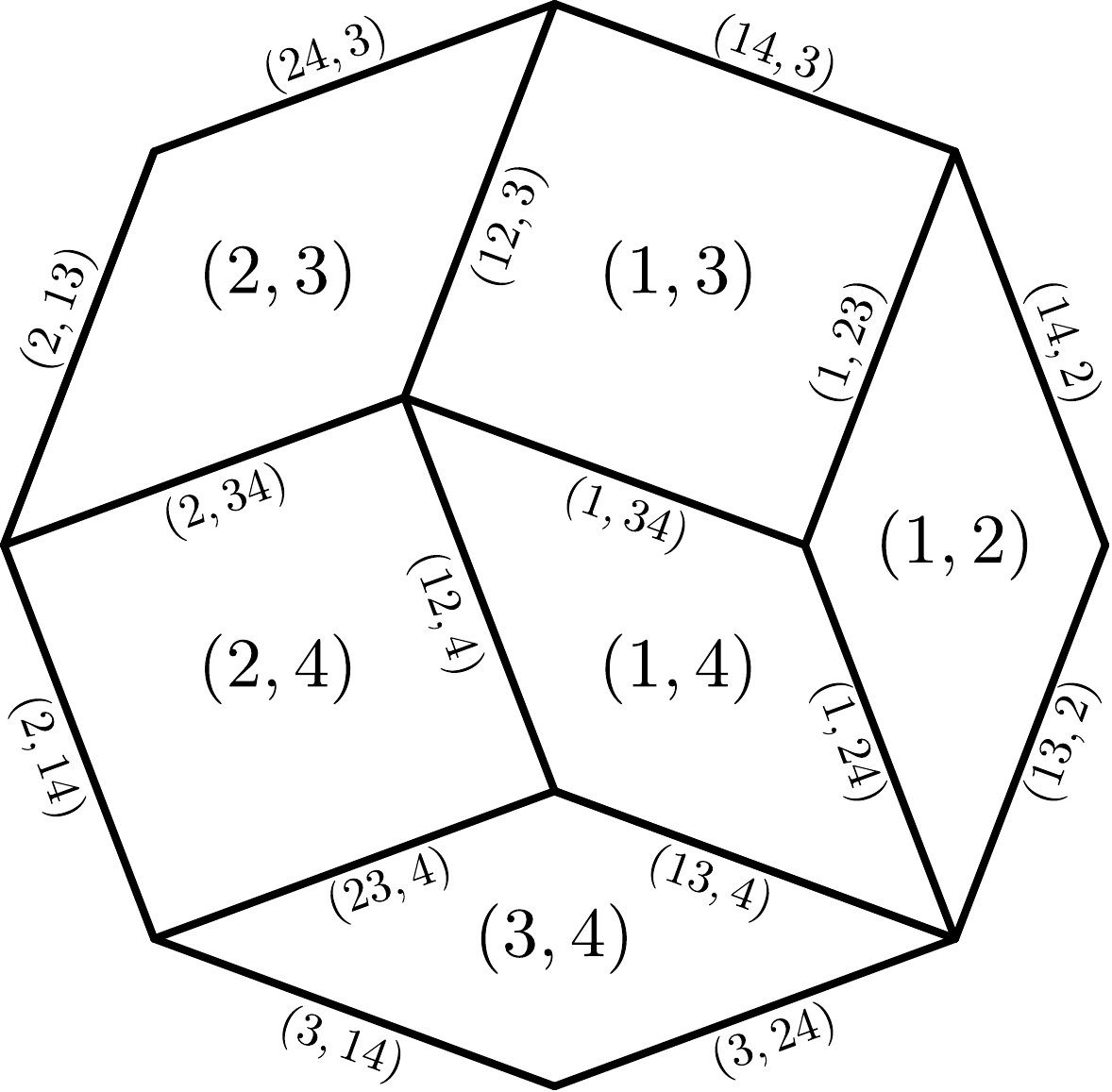}
\qquad
\includegraphics[width=2.5in]{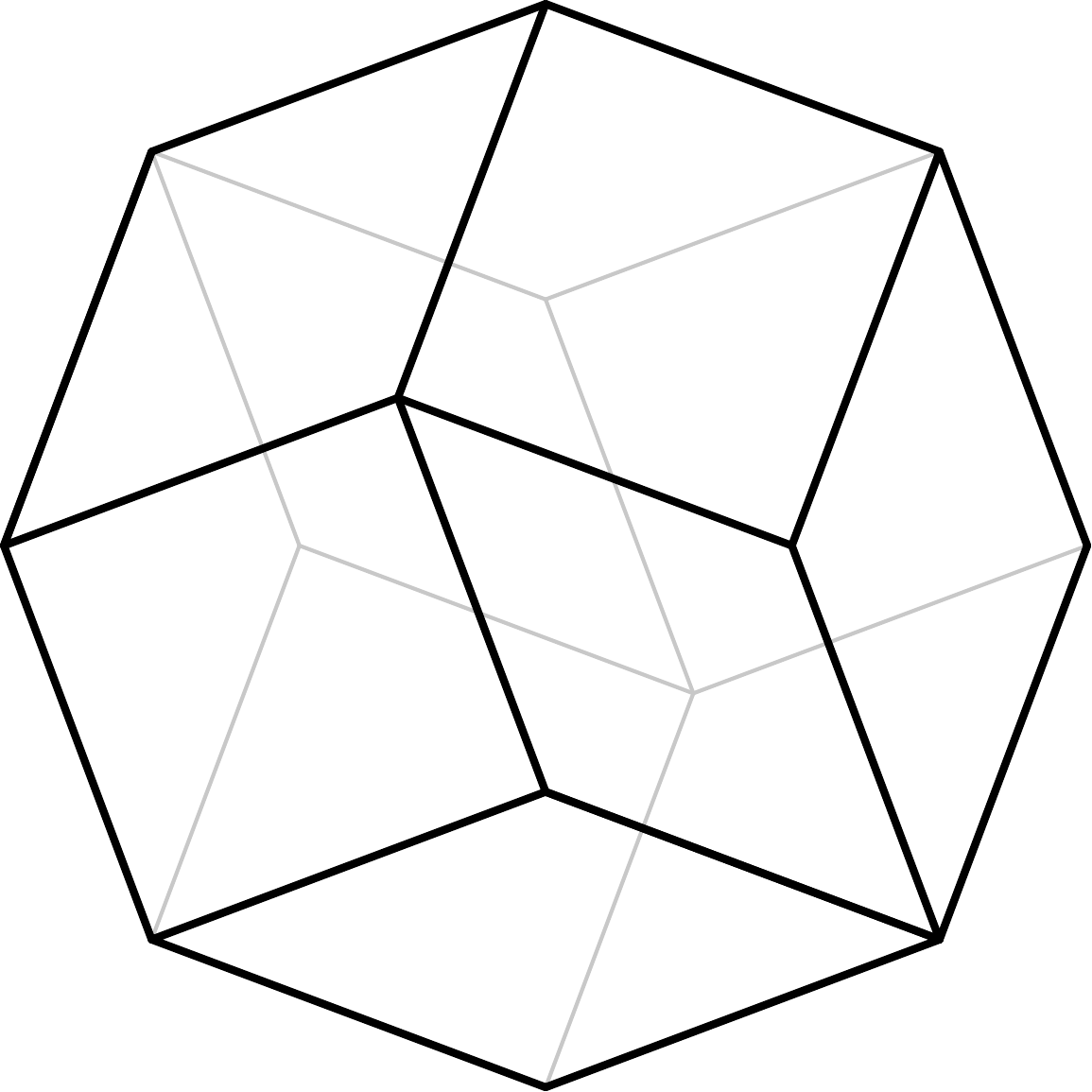}
\caption{The cubical complex corresponding to the poset $(\mathcal{P}_4,\preceq)$ (left) and the zonotope of which it is a projection under Corollary~\ref{c:real_projective_space} (right).}
\label{fig:cubical_complex}
\end{figure}

We show in Theorem~\ref{t:zonotope} that $P_k$, viewed as a cubical complex, coincides with the lattice of proper faces of a polytope (one whose proper faces are all necessarily cubes).  The same cannot be said for $\mathcal P_k$, as Corollary~\ref{c:real_projective_space} implies its realization as a cubical complex has the homotopy type of real projective space, rather than that of a sphere.  

Following~\cite[Chapter~7]{ziegler}, a \emph{zonotope} is the image of a $d$-cube under affine projection; that is, a polytope of the form 
$$Z(\mathbf V) = \mathbf{V}\cdot C_{d} \subseteq \RR^p$$
for some vector configuration $\mathbf{V} \in \RR^{p \times d}$ (encoded as the columns of a matrix). 
The~\emph{value vectors} of $\mathbf{V}$ are the elements of the set
$$\Val(\mathbf{V}) = \set{\mathbf{c}\mathbf{V} : \mathbf{c} \in (\RR^p)^*},$$
and the \emph{signed covectors} of $\mathbf{V}$ are the elements of the set
$$
\mathcal{V}^*(\mathbf{V}) = \set{\sign(\mathbf{c}\mathbf{V}) : \mathbf{c} \in (\RR^p)^*} = \SIGN(\Val(\mathbf{V})),
$$
where the entries of $\sign(\mathbf{c}\mathbf{V})$ lie in $\{-,0,+\}$.  

\begin{prop}[{\cite[Corollary~7.17]{ziegler}}]\label{prop:bijective_from_ziegler}
Let $\mathbf{V} \in \RR^{p \times d}$ be a vector configuration in $\RR^p$.  Then there is a natural bijection between the sign vectors of the nonempty faces of the zonotope $Z(\mathbf{V})$ and the signed covectors of the configuration $\mathbf{V}$.  
\end{prop}

\begin{thm}\label{t:zonotope}
The poset $P_k$ is isomorphic to the face lattice of the boundary of the zonotope given by the affine projection $\pi: \RR^{k} \longrightarrow \RR^{k-1}$ with $\pi(\mathbf{x}) = \mathbf{V} \mathbf{x}$, where 
$$
\mathbf{V}
= \begin{pmatrix} I_{k-1} & \mathbf{1} \end{pmatrix}
= \begin{pmatrix}1 & 0 & \cdots & 0 & 1 \\ 0 & 1 & \cdots & 0 & 1 \\ \vdots & \vdots & \ddots & \vdots & \vdots \\ 0 & 0 & \cdots & 1 & 1  \end{pmatrix}.
$$
\end{thm}

\begin{proof}
We proceed via Proposition~\ref{prop:bijective_from_ziegler}.  
The value vectors of $\mathbf{V}$ are those of the form
$$
\mathbf{c}\mathbf{V}
=  \begin{pmatrix} c_1 & \cdots & c_{k-1} & \sum_{i=1}^{k-1}c_i \end{pmatrix}.
$$
for a row vector $\mathbf{c} = \begin{pmatrix} c_1 & \cdots & c_{k-1} \end{pmatrix} \in (\RR^{k-1})^{*}$.  
Observe that the first $d-1$ coordinates in $\mathbf{cV}$ can be any value $c_i \in \RR$ , and so can take on any sign value. The sign of the $k$th coordinate is restricted to the possible signs that result from summing the first $k-1$ coordinates.  Any nonzero sign vector $\mathbf{u} \in \mathcal{V}^*(\mathbf{V})$ thus falls into one of two types:
\begin{enumerate}[(i)]
\item 
$u_k = 0$ and $\mathbf{u}$ has at least one + and one - in the first $k-1$ entries; or

\item 
$u_k \ne 0$ and $u_i = u_k$ for some $i$.  

\end{enumerate}
For each $\mathbf{u} \in \mathcal{V}^*(\mathbf{V})$, define
$$
V(\mathbf{u})^- = \set{i \in [k]: v_i = -}
\qquad \text{and} \qquad
V(\mathbf{u})^+  = \set{i \in [k]: v_i = +},
$$
which together uniquely determine $\mathbf{u}$.  Let $f:P_k \to \mathcal{V}^*(\mathbf{V})$ be given by $f(\hat 0) = \mathbf{0}$ and $f(R,T) = \mathbf{v}$, with
$$
V(\mathbf{v})^- = \begin{cases}
R & \text{if } k \not\in R \cup T; \\
R \sm \set{k} & \text{if } k \in R; \\
R \cup \set{k} & \text{if } k \in T,
\end{cases}
\qquad \text{and} \qquad
V(\mathbf{v})^+ = \begin{cases}
T  & \text{if } k \not\in R \cup T; \\
T \cup \set{k} & \text{if } k \in R; \\
T \sm \set{k} & \text{if } k \in T,
\end{cases}
$$
Notice if $k \notin R \cup T$, then $\mathbf{v}$ is type~(i) since $R, T \ne \emptyset$, and if $k \in R$, then $|V(\mathbf{v})^+| \ge 2$, so $\mathbf{v}$ is type~(ii).  A symmetric argument for $k \in T$ ensures $f(R,T) \in \mathcal{V}^*(\mathbf{V})$.  

Now, $f$ is a bijection,
with $f^{-1}(\mathbf{v}) = (R,T)$ given by 
$$
R = \begin{cases}
V(\mathbf{v})^- & \text{if } v_k = 0; \\
V(\mathbf{v})^-\cup\set{k} & \text{if } v_k = +; \\
V(\mathbf{v})^-\sm\set{k} & \text{if } v_k = -,
\end{cases}
\qquad \text{and} \qquad
T = \begin{cases}
 V(\mathbf{v})^+ & \text{if } v_k = 0; \\
 V(\mathbf{v})^+\sm\set{k} & \text{if } v_k = +; \\
 V(\mathbf{v})^+\cup\set{k} & \text{if } v_k = -,
\end{cases}
$$
for each nonzero $\mathbf{v} \in \mathcal{V}^*(\mathbf{V})$.  It remains to prove $f$ induces an order-preserving map from $P_k$ to the face lattice of the zonotope  $Z(\mathbf{V})$.  To this end, supposing $(R,T) \preceq (I, J)$ and letting $\mathbf{v} = f(R,T)$ and $\mathbf{u} = f(I,J)$, we must show $u_i \in \{0, v_i\}$ for each $i$.  For each $i < k$, we indeed have $u_i \in \{0,v_i\}$ since $I \subseteq R$ and $J \subseteq T$.  If $k \notin R \cap T$, then $u_k = v_k = 0$; if $k \in R \setminus I$, then $u_k = 0$ and $v_k = +$; and if $k \in I$, then $u_k = v_k = +$.  
A~symmetric argument for the case $k \in T$ completes the proof.  
\end{proof}

\begin{cor}\label{c:real_projective_space}
When viewed as a cubical complex, $\mathcal{P}_k$ is homeomorphic to the real projective space $\RR\mathbb{P}^{k-2}$.  \qed
\end{cor}

\begin{proof}
This follows from the fact that $\mathcal{P}_k$ is homeomorphic to $\mathbb S^{k-2}$ by Theorem~\ref{t:zonotope}, and the natural map $P_k \to \mathcal P_k$ sending $(I,J) \mapsto \{I,J\}$ induces a 2-to-1 covering map on cubical complexes that identifies opposite faces of the zonotope in Theorem~\ref{t:zonotope}.  
\end{proof}


\end{document}